\documentclass{amsart} 
\usepackage{amssymb,latexsym,textcomp}
\usepackage{bm}
\usepackage{xcolor,graphicx}
\usepackage{multirow}
\usepackage{cite} 

\usepackage{hyperref}
\hypersetup{colorlinks,allcolors=blue}
\usepackage[OT2,T1]{fontenc}

\title[Additive preservers of mutual strong BJ orthogonality]{Additive preservers of mutual strong Birkhoff-James orthogonality on finite-dimensional $C^\ast$-algebras}

\author{Bojan Kuzma}
\thanks{This work is supported in part by the Slovenian Research Agency (research program P1-0285 and research projects J1-50000,
N1-0296,
N1-0428,
J1-70047, and
J1-70046) and by the Ministry of Science, Technological Development and Innovation of Republic of Serbia: grant number 451-03-47/2023-1/200104 with Faculty of Mathematics.}
\address{University of Primorska, Glagolja\v{s}ka 8, SI-6000 Koper, Slovenia, and Institute of Mathematics, Physics, and Mechanics, Jadranska 19, SI-1000 Ljubljana, Slovenia}
\email{bojan.kuzma@upr.si}

\author{Srdjan Stefanovi\'c}
\address{University of Belgrade\\ Faculty of Mathematics\\ Student\/ski trg 16-18\\ 11000 Beograd\\ Serbia}
\email{srdjan.stefanovic@matf.bg.ac.rs}



\newtheorem{theorem}{Theorem}[section]
\newtheorem{lemma}[theorem]{Lemma}

\newtheorem{corollary}[theorem]{Corollary}

\theoremstyle{definition}

\newtheorem{remark}[theorem]{Remark}
\newtheorem{example}[theorem]{Example}

\newcommand{\A}{\mathcal A}

\newcommand{\C}{\mathbb C}

\newcommand{\strongperp}{\mathrel{\perp\!\!\!\!\perp}^S}
\newcommand{\strongp}{\mathrel{\perp}^S}
\newcommand\rk{\mathop{\mathrm{rk}}}

\def\CC{\mathbb C}
\def\FF{\mathbb F}

\def\rank{\mathop{\mathrm{rk}}}

\def\cV{\mathcal V}

\numberwithin{equation}{section} 

\begin{document}

\begin{abstract}
    We describe additive surjections on direct sum of matrix algebras that preserve singularity in one direction. As an application, we classify additive surjections on finite-dimensional $C^{\ast}$-algebras that preserve mutual strong Birkhoff–James orthogonality in one direction.
\end{abstract}

\maketitle

\maketitle

\section{Introduction, preliminaries and statement of the main results}

In a normed space $X$ we lack an inner product; therefore, a notion of orthogonality is required that coincides with the standard definition
$$x \perp y \Leftrightarrow \langle x,y\rangle = 0$$
whenever $X$ is an inner product space (throughout this paper, the scalar field is assumed to be the field of complex numbers $\mathbb{C}$, unless explicitly mentioned otherwise). One of the most studied notions is the Birkhoff--James orthogonality, which has been investigated for more than 90 years (see the pioneering works of Birkhoff \cite{Birkhoff} and James \cite{James1},\cite{James2},\cite{James3}).

We say that two vectors $x,y \in X$ are Birkhoff--James (BJ for short) orthogonal, and write $x \perp_{BJ} y$, if
$$\|x + \lambda y\| \geq \|x\|, \quad \text{for all } \lambda \in \mathbb{C}.$$

In \cite{ArambasicRajic2014}, Aramba\v{s}i\'c and Raji\'c naturally extended the above definition to (right) Hilbert $C^{\ast}$-modules over a $C^{\ast}$-algebra $\mathfrak{A}$. Namely, for two elements $x,y \in X$ we say that $x$ is strongly BJ orthogonal to $y$, and write $x \perp_{BJ}^{S} y$, if
$$\|x + yc\| \geq \|x\| \quad \text{for all } c \in A.$$
Then, in \cite{ArambasicBJMA2020}, Aramba\v{s}i\'c et al. defined the mutual version of the relation previously considered. Namely, for two elements $x, y \in X$, we say that $x$ and $y$ are mutual strong BJ orthogonal, denoted by $x \strongperp_{BJ}\, y$, if $x \perp_{BJ}^{S} y$ and $y \perp_{BJ}^{S} x$.

In this paper, we will consider a $C^{\ast}$-algebra $\mathfrak{A}$, regarded as a right Hilbert $C^{\ast}$-module over itself. So $a\in \mathfrak{A}$ is strong BJ orthogonal to $b\in\mathfrak{A}$ if for any $c\in \mathfrak{A}$ there holds
\begin{equation}\label{strong BJ}
\|a+bc\|\ge\|a\|.
\end{equation}

The essence of this relation can be studied by the orthograph introduced in \cite{ArambasicBJMA2020}. Recall that the vertices of this orthograph are all elements of the projective space over $\mathfrak{A}$ , i.e., $\{[a]=\CC a;\;\; a\in\mathfrak{A}\setminus\{0\}\}$, and two vertices $[x]$ and $[y]$ form an edge if $x\strongperp_{BJ}\, y$ for some representatives $x\in[x]$ and $y\in[y]$. In \cite{ArambasicBJMA2020} its  diameter and isolated vertices were determined in case when $\mathfrak{A}=B(H)$, where $H$ is Hilbert space, and also in  case when  $\mathfrak{A}$ is a  commutative $C^{\ast}$-algebra. This was essentially used to determine linear preservers of mutual strong BJ orthogonality on $B(H)$ (see \cite{StrongMutualPreservers}).  Later, in \cite{Keckic2023JMAA}, isolated vertices in an arbitrary $C^{\ast}$-algebra are described, and an estimate of the diameter is given. Finally, in \cite{Stefanovic}, this orthograph was considered when $\mathfrak{A}$ is a finite-dimensional $C^{\ast}$-algebra, i.e.,
\begin{equation}\label{finite-dimensional}
    \mathfrak{A}\cong M_{n_1}(\CC)\oplus M_{n_2}(\C)\oplus\dots\oplus M_{n_k}(\C),
\end{equation}
for some natural numbers $k,n_1,n_2,\dots,n_k$ up to permutation, and the diameter was determined. To avoid confusion, in the two aforementioned papers the orthograph on the projective space was considered, whereas here our ambient space will be the entire $C^{\ast}$-algebra. We will repeatedly use a strong version of the well-known Stampfli-Magajna-Bhatia-\v{S}emrl theorem (for the proof see \cite[Proposition 2.8]{ArambasicRajic2014}; see also papers ~\cite{STAMPF}, ~\cite{MAGAJN} and ~\cite{Bhat-Semr} for a retrospective).

\begin{theorem}\label{THEOREM}
    Let $A,B\in M_n(\C)$. Then $A\perp_{BJ}^{S}\,B$ if and only if there is a unit vector $x\in\C^n$ such that $\|Ax\|=\|A\|$ and $B^{\ast}Ax=0$.
\end{theorem}

In this paper, we study additive preservers. They have been studied for many years (see, for example, the papers ~\cite{OmladicSemrl} and ~\cite{KuzmaAdditive}). Recently, there was a renewed interest in additive preservers of the orthogonality relation --- see the paper of W\'ojcik  \cite[Section 4]{Wojcik}, where he proved that the additive preserver of BJ orthogonality on real normed space is automatically linear and hence, from the famous result by Oman~\cite[Lemma 3.3, page 38]{Oman}, it is a scalar multiple of an isometry. We remark in passing that Blanco and Turn\v{s}ek \cite{BlancoTurnsek2006} 
extended this result by classifying bijections preserving BJ orthogonality in both directions on projective spaces induced by reflexive smooth spaces.
 Also, in two papers by Li, Liu, and Peralta (see \cite{LiLiuPeralta} and \cite{LiLiuPeralta2}), W\'ojcik's result was extended to complex normed spaces. In \cite[Proposition 2.2]{LiLiuPeralta} it was shown that an additive map $\Phi$, which preserves the BJ orthogonality between two complex normed  spaces $X,Y$, with $\dim X\ge 2$ and admitting a conjugation (i.e., there exists a conjugate-linear and involutive isometry on $X$) is automatically real-linear; if, in addition, $\Phi$ is surjective and preserves BJ orthogonality in both directions, then it is automatically continuous. This result was extended in \cite[Lemma 2.14]{LiLiuPeralta2} where it was shown that $\Phi$ admits a unique real-linear extension to a BJ preserver on completions of $X$ and $Y$ provided that, besides the previous assumptions on $X$, we also have that $X^\ast$ is strictly convex.

Our main contribution here is the following result. Recall that, for a given orthonormal basis $e_1,\dots,e_n$ of $\CC^n$, the map
$$J\colon x=\sum \langle x,e_i\rangle e_i\mapsto \overline{x}:=\sum \overline{\langle x,e_i\rangle }e_i$$
is a conjugate-linear isometry on $\CC^n$, and it induces a conjugate-linear isometry on a $C^\ast$-algebra $M_n(\CC)$ by transforming a matrix $X\in M_n(\CC)$ to $$X\mapsto JXJ.$$
It is immediate that $JXJ=\overline{X}$ is an (entrywise) conjugation of the elements of $X$.
\begin{theorem}\label{THM:Main}
  Let $\mathfrak{A}$ be a finite-dimensional complex $C^\ast$-algebra, not isomorphic to $\CC$, $\CC\oplus\CC$ and $M_2(\CC)$, and let $\Phi\colon\mathfrak{A}\to\mathfrak{A}$ be an additive surjection.  Then, the following are equivalent
  \begin{itemize}
      \item[(i)] $\Phi$ preservers mutual strong BJ orthogonality.
      \item[(ii)] There exists a permutation $\pi$ of minimal ideals of $\mathfrak{A}$, which respects their dimension, a positive scalar $\gamma$ and unitary elements $u,v\in\mathfrak{A}$ such that
      $$\Phi(a)=\gamma u\pi(a)^\dagger v,$$
      where $x\mapsto x^\dagger$  denotes a map which is identity on some (possibly none) minimal ideals and entrywise conjugation with respect to a fixed orthonormal basis on other minimal ideals.
     \end{itemize}
\end{theorem}
We remark that the structure of additive preservers of mutual strong BJ orthogonality on $\CC$, $\CC\oplus\CC$ or $M_2(\CC)$,
is different; we refer to the last section for more details.

Clearly, each map which preserves strong BJ orthogonality also preserves mutual strong BJ orthogonality. Hence, as an immediate corollary, we can classify additive preserves of strong BJ orthogonality:
\begin{corollary}
Under the assumptions of Theorem~\ref{THM:Main}, an additive surjection $\Phi\colon\mathfrak{A}\to\mathfrak{A}$ preserves strong BJ orthogonality if and only if it equals one of the forms in item (ii) if Theorem~\ref{THM:Main}. 
\end{corollary}

Our main tool will be a reduction to preservers of singularity. Since this is of independent interest, we shall state it here, after clarifying the following terminology: Any permutation $\pi$  of $\{1,.\dots,k\}$  induces a linear bijection, again denoted by $\pi$,  on a block-diagonal matrix algebra $\A=\bigoplus_{1}^k M_{n_i}(\FF)$, given by $\bigoplus A_i\mapsto \bigoplus A_{\pi(i)}$. For example, the transposition $\tau=(1,2)$ will map $A\oplus B$ into $B\oplus A$.  

\begin{theorem}\label{singularity_preservers}
Let $\FF$ be a field of characteristic zero and let $k,n_1,\dots,n_k\ge 1$ be integers. The following is equivalent for an additive surjection $$\Phi\colon M_{n_1}(\FF)\oplus\dots\oplus M_{n_k}(\FF)\to
 M_{n_1}(\FF)\oplus\dots\oplus M_{n_k}(\FF).$$
\begin{itemize}
    \item[(i)] $\Phi$ maps singular elements into singular elements. 
    \item[(ii)] There exists a permutation of blocks $\pi$, which respects their dimension, such that 
    $$\Phi(M_{n_i}(\FF))=\pi(M_{n_i}(\FF)).$$ Moreover, if $n_i\ge 2$ there exist a field endomorphism $\sigma_i\colon\FF\to\FF$ and invertible matrices $P_i,Q_i\in M_{n_i}(\FF)$,  $(i=1,\dots,k)$, such that 
    $$\Phi\colon A_i\mapsto P_i\pi(A_i)^{\sigma_i} Q_i\quad \hbox{ or }\quad \Phi\colon A_i\mapsto P_i\pi(A_i^T)^{\sigma_i} Q_i;\qquad  A_i\in M_{n_i}(\FF).$$
\end{itemize} 
\end{theorem}

\section{Additive preservers of singularity}

Within this section, we prove Theorem~\ref{singularity_preservers}. We remark that the  main idea, which is a reduction to additive rank-one nonincreasing maps, comes from a paper by Fo\v{s}ner and \v{S}emrl,~\cite{FosnerSemrl}. Throughout this section, 
$$\A=
M_{n_1}(\FF)\oplus M_{n_2}(\FF)\oplus\dots\oplus M_{n_k}(\FF)$$
is a block-diagonal algebra over the field $\FF$, with $\mathrm{char}(\FF)=0$.

\begin{lemma}\label{finite_many_singular}
    Let $m_1,m_2\ge1$ be integers and $A_{22}\in\mathrm{GL}_{m_2}(\FF)$. 
    Then $$\begin{pmatrix}
        pI_{m_1}+A_{11} & A_{12}\\
        A_{21} & A_{22}
    \end{pmatrix}\in M_{m_1+m_2}(\FF)$$ is invertible except for at most finitely many values $p \in \mathbb{N}=\{1,2,\dots\}$. 
\end{lemma}

\begin{proof}
    Observe that $\det\begin{pmatrix}
        pI_{m_1}+A_{11} & A_{12}\\
        A_{21} & A_{22}
    \end{pmatrix}$ is a polynomial in $p$ of degree $m_1$. Moreover, the coefficient of $p^{m_1}$ is equal to $\det(A_{22})$ which is nonzero. Hence, it has at most $m_1$ zeros (some of which may lie outside $\mathbb{N}$), and the claim follows.
\end{proof}

\begin{lemma}\label{main_singular}
    Let $\Phi\colon\A\to\A$ be an additive, surjective map that preserves singularity in one direction. Then $\Phi$ maps rank-one operators to rank-one operators or to zero.
\end{lemma}
\begin{proof}
    Suppose, to the contrary, that there exists a rank-one matrix in $\A$ whose $\Phi$-image is a matrix of rank $\ell$ ($2\leq \ell<N:=n_1+n_2+\dots+n_k$) in $\A$. 
    By pre- and post- composing  $\Phi$ with suitable transformations of the form $A \mapsto PAQ$, where $P$ and $Q$ are invertible block-diagonal matrices (such map preserves singularity), we may assume that $$\Phi(E_{11})=\bigoplus_{i=1}^{k}\left(I_{n_i'}\oplus 0_{n_i-n_i'}\right),$$
    where $n_i'$ can be zero, with at least one $n_i'\ge1$.
    Define $F=I-\Phi(E_{11})$ and a (compression) subspace $\mathcal{V}=F\A F|_{R(F)}$; notice that
\begin{equation}\label{eq:V}
\cV=\bigoplus_{n'_i\ge1} M_{n'_i}(\FF)\subseteq M_{N-\ell}(\FF)\subseteq M_{N-2}(\FF).
\end{equation}

Let $T$ be an arbitrary element in $\A$. We will show that the compression of $\Phi(T)$ on $\mathcal{V}$ is singular, which contradicts the surjectivity of $\Phi$.

If $\rk(T)\leq N-2$, then $pE_{11}+T$ is singular for every $p\in\mathbb{N}$ (by rank argument). By additivity and singularity preserving, $p\Phi(E_{11})+\Phi(T)$ is also singular for every $p\in\mathbb{N}$. Assume that $\Phi(T)$ is invertible in $\mathcal{V}$. Regarding the blocks of $\A$ we have three possibilities: $$n_i'=0, \quad n_i'=n_i,\ \hbox{ and }\quad 0<n_i'<n_i.$$

In a block with $n_i'=0$ the matrix $\Phi(pE_{11}+T)|_{M_{n_i}(\FF)}=pI_{n_i} + \Phi(T)|_{M_{n_i}(\FF)}$ is singular if and only if $-p$ is an eigenvalue of $ \Phi(T )|_{M_{n_i}(\FF)}$, and there are at most finitely many such integers $p$.

A block with $n_i'=n_i$ belongs completely to $\cV$, and by the assumption, $\Phi(pE_{11}+T)|_{M_{n_i}(\FF)}=0_{n_i}+\Phi(T)|_{M_{n_i}(\FF)}$ is invertible there for each integer $p$.

Finally in a block with $0<n_i'<n_i$ we get that $\Phi(T)$ is invertible except for at most finitely many values $p\in\mathbb{N}$ by Lemma \ref{finite_many_singular}.

Combined, this shows that, except for at most finitely many integers $p$, all blocks of a singular matrix $\Phi(p E_{11}+T)$ are invertible, a contradiciton. Indeed, the compression of $\Phi(T)$ to $\cV$ must be singular whenever $\rank T\le N-2$.

Next, if $\rk(T)=N-1$, write $T$ as a sum of $N-1$ rank-one matrices $T_1, \dots,T_{N-1}$ from $\A$. Then, by the preceding result, for any sum of $N-2$ or fewer elements of the set  $\{T_1,T_2,\dots,T_{N-1}\}$, its $\Phi$-image is singular in $\mathcal{V}$. By \eqref{eq:V} we may consider $\mathcal{V}$ to be embedded into $M_{N-2}(\FF)$, and  applying \cite[Lemma 2.1]{FosnerSemrl} (for $m=N-1>n=N-2$) we get $\Phi(T)$ is singular in $\mathcal{V}$. 

If $\rk(T)=N$  we write $T$ as a sum of $N$ rank-one matrices $T_1,  \dots, T_N\in\A$, and apply the previous arguments to show that $\Phi(T)$ is again singular in $\cV$.
\end{proof}

\begin{lemma}\label{blocks_to_blocks}
    Let $\Phi\colon \A\to\A$ be an additive, surjective map that preserves singularity in one direction. Then $\Phi$ maps each block $M_{n_i}(\FF)$ onto some block $M_{n_j}(\FF)$. 
\end{lemma}
\begin{proof}
    Since $\Phi$ is an additive surjection, there exists a rank-one element $R\in\A$ with $\Phi(R)\neq0$. By Lemma \ref{main_singular} $\rk\,\Phi(R)=1$, so there exists $j\in[1,k]$ such that $\Phi(R)\in M_{n_j}(\FF)$. Let $i$ be the index of the block containing $R$. Denote $$R=x\otimes y\quad \text{\ where\ }\quad x,y\in0\oplus\FF^{n_i}\oplus0;$$  
    (here and throughout, given $x,y\in\FF^n$ we denote by $x\otimes y$ a rank-one matrix, which, in a standard basis, maps $z$ into $(y^\ast z)x$, where $y^\ast$ denotes a conjugated transpose of a column vector $y$).
    Now, if for some $w\in0\oplus\FF^{n_i}\oplus 0$ it holds $0\neq\Phi(x\otimes w)\in M_{n_\ell}(\FF)$ where $n_\ell\neq n_j$, then $x\otimes y+x\otimes w$ is rank-one, however $\Phi(x\otimes y+x\otimes w)$ is rank two, a contradiction. The same argument shows that $\Phi(z\otimes y)\in  M_{n_j}(\FF)$ for all $z\in 0\oplus\FF^{n_i}\oplus0$. 
    
    Now, take an arbitrary rank-one $z\otimes w\in M_{n_i}(\FF)$ with $z,w\in0\oplus\FF^{n_i}\oplus0$. Then, by the above,  
    \begin{equation}\label{eq:xw--zy}
        \Phi(x\otimes w),\Phi(z\otimes y)\in M_{n_j}(\FF).
    \end{equation}
    If at least one of them is nonzero, say $\Phi(x\otimes w)\neq0$, we can repeat the  above arguments, to deduce that also    $\Phi(z\otimes w)\in M_{n_j}(\FF)$. However, if both elements of \eqref{eq:xw--zy} are zero, then consider a rank-one $(x+z)\otimes (y+w)$, which $\Phi$ maps into $\Phi(x\otimes y)+0+0+\Phi(z\otimes w)$ again of rank-one at most. Since $0\neq \Phi(x\otimes y)\in M_{n_i}(\FF)$, we necessarily have that  $\Phi(z\otimes w)\in M_{n_j}(\FF)$. It now follows by the additivity that the whole block $M_{n_i}(\FF)$ is mapped into $M_{n_j}(\FF)$.
    
   Thus, $\Phi$ maps each block $M_{n_i}(\FF)\subseteq\A$ into some block $M_{n_j}(\FF)$. Since there are finitely many blocks, it must permute the blocks and must map $M_{n_i}(\FF)$ onto $M_{n_j}(\FF)$  otherwise it would not be surjective.  
\end{proof}

\begin{proof}[Proof of Theorem \ref{singularity_preservers}]
    
Using the Lemma \ref{blocks_to_blocks}, we can denote 
$$\Phi_i\colon =\Phi|_{M_{n_i}(\FF)}\colon M_{n_i}(\FF)\to M_{n_j}(\FF).$$ Now we are in a position to apply the result \cite[Theorem 2; Corollary 3]{Lim1} to $D=D_1=\FF$ and $m=n=n_i, p=q=n_j$. Cases (i) and (ii) of \cite[Corollary 3]{Lim1}  are excluded by surjectivity if $n_j\ge2$ (see also \cite{Lim2} for more results in this direction). Therefore, when $n_j\ge2$, we conclude that the only remaining  possibilities are 

(iii) $\Phi_i(A)=P_iA^{\sigma_i}Q_i$  and 

(iv) $\Phi_i(A)=(P_iA^{\sigma_i} Q_i)^T$,

\noindent
where $P_i\in M_{n_j,n_i}(\FF)$ is surjective and $Q_i\in M_{n_i,n_j}(\FF)$ is injective (by surjectivity of $\Phi_i$) and $\sigma_i\colon \FF\to\FF$ is a field isomorphism. 

We claim $\Phi_i$ maps the block $M_{n_i}(\FF)$ to a block of the same size. Indeed, if $n_i\ge2$, we observe that $\dim\Phi_{i}(M_{n_i}(\FF))\le\dim M_{n_i}(\FF)$ since $\Phi_i$ is semilinear. Consequently, if $\A$ contains a block of size greater than one, then only a block of maximal possible size in $\A$ can be mapped to a block of maximal size in $\A$. Since we know that $\Phi$ permutes the blocks we can continue these arguments to complete the proof.
\end{proof}

\section{Additive preservers of mutual strong Birkhoff-James orthogonality}

We will apply the results of Theorem \ref{singularity_preservers} in the case $\FF=\C$ to clasify mutual strong BJ orthogonality preservers on finite-dimensional $C^{\ast}$-algebra $\mathfrak{A}$. Recall that they are all of the form \eqref{finite-dimensional}, and recall that $A$ is mutually BJ orthogonal to $B$ if $A\perp^S_{BJ}\, B$ and $B\perp^S_{BJ}\, A$, where the relation of strong BJ orthogonality, $\perp^S_{BJ}$, was defined at \eqref{strong BJ}.

We will frequently use the following three simple observations: (i) if $B\perp^S A$ for an invertible matrix $A$, then $B=0$. Namely, if we put $C=-A^{-1}B$ in relation $\|B+AC\|\ge\|B\|$ we immediately get $B=0$. See also \cite[Lemma 1.2]{ArambasicBJMA2020} and \cite[Proposition 2.5 and 2.6]{Keckic2023JMAA}. Also, (ii) $A\strongp_{BJ} A$ if and only if $A=0$. Finally, (iii)  for a rank-one operator $A$, the relation $A\strongperp_{BJ}\,B$ is equivalent to $A^{\ast}B=0$ (see \cite[Proposition 2.3]{ArambasicStrong}).

\begin{lemma}\label{singula_to_singular}
    Let $\mathfrak{A}$ be a finite-dimensional $C^{\ast}$-algebra not isomorphic to $\mathbb{C}\oplus\mathbb{C}$. Furthermore, let $\Phi\colon\mathfrak{A}\to\mathfrak{A}$ be an additive map that preserves a mutual strong BJ orthogonality in one direction. Then $\Phi$ preserves the singularity.
\end{lemma}
\begin{proof}
    If $\mathfrak{A}=\C$, the statement of Lemma \ref{singula_to_singular} is trivially valid. Indeed, the only singular element in $\mathbb{C}$ is 0, and $\Phi(0)=0$. 

    If $\mathfrak{A}=M_2(\C)$, suppose, to the contrary, that there exists a nonzero singular element (and thus of rank-one, which we denote by $x\otimes y$) that is mapped into an invertible element. Applying $\Phi$ to the relation 
    $$x\otimes y\strongperp_{BJ}\,x'\otimes z$$ (here and throughout, $x'$ denotes any vector perpendicular to $x$, with $\|x'\|=\|x\|$) we obtain $\Phi(x'\otimes z)=0$ for an arbitrary $z\in\C^2$, because $\Phi(x\otimes y)$ is invertible. Then by Theorem \ref{THEOREM} 
    $$(x+x')\otimes y\strongperp_{BJ}\,(x- x')\otimes y,$$
    so applying additive $\Phi$ on both sides, we get
    $$\Phi(x\otimes y)\strongperp_{BJ}\,\Phi(x\otimes y),$$
    which is impossible by invertibility of $\Phi(x\otimes y)$ (in fact, it would imply $\Phi(x \otimes y)=0$). 
    
    In all other cases, we can embed $\mathfrak{A}$ into $M_N(\mathbb{C})$ where $N=n_1+n_2+\dots+n_k$ and show that $\Phi(T)$ is singular whenever $T$ is singular. First, suppose that $\rk\,T\leq N-2$. Let $$T=\sigma_1x_1\otimes y_1+\sigma_2x_2\otimes y_2+\dots+\sigma_{N-2}x_{N-2}\otimes y_{N-2},\ \sigma_1\geq\sigma_2\geq\dots\geq\sigma_{N-2}\geq0,$$
    be its  singular value decomposition (SVD for short) and let us assume the opposite, that $\Phi(T)$ is invertible. We extend the set $\{x_1,\dots,x_{N-2}\}$ to an orthonormal basis $\{x_1,\dots,x_{N-2},x_{N-1},x_N\}$ and similarly with $y_i$'s. Then, by Theorem \ref{THEOREM}, we know that 
    $$T\strongperp_{BJ}\,x_i\otimes y_i,\ \text{where\ }i\in\{N-1,N\},$$
    and applying $\Phi$ on both sides, we get $\Phi(T)\strongperp_{BJ}\,\Phi(x_i\otimes y_i)$ and therefore  $\Phi(x_i\otimes y_i)=0,\ i\in\{N-1,N\}$ because $\Phi(T)$ is invertible. Also, again by Theorem~\ref{THEOREM}
    $$T+p\,x_{N-1}\otimes y_{N-1}\strongperp_{BJ}\, T+p\, x_N\otimes y_N,$$
    for $p\in\mathbb{N}$ greater than $\sigma_1$. Applying additive $\Phi$  we get $\Phi(T)\strongperp_{BJ}\,\Phi(T)$, so $\Phi(T)=0$, which contradicts the assumed invertibility of $\Phi(T)$. 

    Finally, suppose $\rk T=n-1$. Again, let $$T=\sigma_1x_1\otimes y_1+\sigma_2x_2\otimes y_2+\dots+\sigma_{N-1}x_{N-1}\otimes y_{N-1},\ \sigma_1\geq\sigma_2\geq\dots\geq\sigma_{N-1}>0,$$
    be its SVD and let us assume the opposite, that $\Phi(T)$ is invertible. Then, by Theorem \ref{THEOREM} 
    $$T\strongperp_{BJ}\, x_N\otimes y_N \pm x\otimes y,$$
    where $x\perp \{x_N,x_1\}$ and $y\perp y_N$ are both nonzero (they exist because $N\ge3$). From the invertibility of $\Phi(T)$, we conclude that $\Phi(x_N\otimes y_N \pm x\otimes y)=0$ and then from the additivity $\Phi(x_N\otimes y_N)=0$ and $\Phi(x\otimes y)=0$. In particular, by taking $x=\sigma_ix_i$ and $y=y_i$ for $2\le i\le N-1$, we conclude that $\Phi(\sigma_ix_i\otimes y_i)=0$. So  $$\Phi(\sigma_1x_1\otimes y_1)=\sum\limits_{i=1}^{N-1}\Phi(\sigma_i x_i\otimes y_i)=\Phi(T)$$ is invertible, which is impossible by the previous part of our proof.
\end{proof}

As we show in the next example, Lemma \ref{singula_to_singular} does not hold in case $\mathfrak{A}=\C\oplus\C$. 
\begin{example}\label{exam:32}
    Let us define $\Phi\colon\C\oplus\C\to\C\oplus\C$ as follows: $$\Phi(\lambda,\mu)=\lambda(1,1).$$ 
   By definition $(\lambda_1,\mu_1)\strongperp_{BJ}\,(\lambda_2,\mu_2)$ for two nonzero elements in $\mathfrak{A}$ if and only if $\lambda_1=\mu_2=0$ or $\mu_1=\lambda_2=0$. In both cases, one element is of the form $(0,\mu)$, so it is mapped to $(0,0)$ by $\Phi$. So, $\Phi$ preserves mutual strong BJ orthogonality, however, it maps a singular element $(1,0)$ to invertible $(1,1)$. 
\end{example}

Now, we are in a position to apply Theorem \ref{singularity_preservers} by which
$$\Phi\colon A_i\mapsto P_i\pi(A_i)^{\sigma_i} Q_i\quad \hbox{ or }\quad \Phi\colon A_i\mapsto P_i\pi(A_i^T)^{\sigma_i} Q_i;\qquad  A_i\in M_{n_i}(\FF).$$

The permutation of blocks $\pi$ is clearly a linear isometry, hence preserves the mutual  strong BJ orthogonality. By composing $\Phi$ with its inverse $\pi^{-1}$, we can and will assume that $\pi$ is the identity permutation. Let us now eliminate the second option (namely, transposition) on blocks of dimension at least two.

\begin{lemma}\label{tranpose_impossible}
    The map $\Phi_i\colon M_{n_i}(\C)\to M_{n_i}(\C)$ for $n_i\ge2$, defined by $\Phi_i(A)=P_i\left(A^T\right)^{\sigma_i}Q_i$ does not preserve the mutual strong BJ orthogonality in one direction.
\end{lemma}
\begin{proof}
    This is essentially proven in \cite[Lemma 2.6]{StrongMutualPreservers}, but we reprove it for the sake of convinience. Let $E$ be any rank-one orthogonal projection with rational entries. Then $(E^{T})^{\sigma_i}=E$. Also,  $E(I-E)=0$, so $E\strongperp_{BJ}\,(I-E)$, which implies that $P_i(E^{T})^{\sigma_i}Q_i\strongperp_{BJ}\, P_i((I-E)^T)^{\sigma_i}Q_i$. This  simplifies into 
    $$P_iEQ_i\strongperp_{BJ}\, P_i(I-E)Q_i.$$ Since $P_iEQ_i$ is a rank-one, we conclude $(P_iEQ_i)^{\ast}P_i(I-E)Q_i=0$, that is,  $$Q_i^{\ast}EP_i^{\ast}P_i(I-E)Q_i=0.$$
    Since $Q_i$ is invertible, $EP_i^{\ast}P_i(I-E)=0$, that is, 
    $$EP_i^{\ast}P_i=EP_i^{\ast}P_iE$$ for all rank-one
    orthogonal projections $E$ with rational coeficients. Insert $E=e_i\otimes e_i$, where $e_1,e_2,\dots,e_{n_i}\in\C^{n_i}$ is a standard orthonormal basis of $\CC^{n_i}$ to deduce that $P_i^{\ast}P_i$ is a diagonal matrix. Insert also $E=\frac{1}{2}(e_1+e_k)\otimes(e_1+e_k)$ to deduce that $P_i^{\ast}P_i$ is a scalar multiple of identity. This implies that the map $$\Psi_i=P_i^\ast P_i\Phi_i:A \mapsto (A^T)^{\sigma_i}Q_i$$
    also preserves a mutual strong BJ orthogonality.
    \noindent
    To finish the proof, note that     
    $$e_1 \otimes e_1\strongperp_{BJ}\,e_2 \otimes (e_1+e_2)$$
    so applying $\Psi_i$ we get  $$(e_1 \otimes e_1)Q_i\strongperp_{BJ}\,((e_1+e_2) \otimes e_2)Q_i.$$ 
    Since both elements are of rank-one, it follows that $Q_i^{\ast}((e_1+e_2)\otimes e_2 )^{\ast}(e_1 \otimes e_1)Q_i=0$. Then invertibility of $Q_i$ implies that $$e_2 \otimes e_1=((e_1 + e_2) \otimes e_2)^\ast(e_1 \otimes e_1)=0,$$
    which is a contradiction.
\end{proof}

\begin{lemma}\label{P_and_sigma}
    If the map $\Phi_i\colon M_{n_i}(\C)\to M_{n_i}(\C),n_i\ge2$ defined by $\Phi_i(A)=P_iA^{\sigma_i} Q_i$ preserves the mutual strong BJ orthogonality in one direction, then $\sigma_i(\lambda)=\lambda$ for all $\lambda\in\C$ or $\sigma_i(\lambda)=\overline{\lambda}$ for all $\lambda\in\C$. Moreover,  $P_i$ is a scalar multiple of a unitary operator. 
\end{lemma}
\begin{proof}
    On rank-one operators, mutual strong BJ orthogonality is equivalent to range orthogonality, so for all $\lambda\in\C$ and $\ell\neq j$ 
$$(e_{\ell}+\lambda e_j)\otimes e_{\ell}\strongperp (\bar{\lambda}e_{\ell}-e_j)\otimes e_{\ell}.$$
Applying $\Phi_i$ on both sides gets
$$P_i(e_{\ell}+\sigma_i(\lambda)e_j)\otimes Q_i^\ast e_{\ell}\strongperp P_i(\sigma_i(\bar\lambda)e_{\ell}-e_j)\otimes Q_i^\ast e_{\ell},$$
from where we can conclude 
\begin{equation}\label{upper 2x2}
(e_{\ell}^\ast+\overline{\sigma_i(\lambda)}e_j^\ast)P_i^{\ast}P_i(\sigma_i(\bar\lambda)e_{\ell}-e_j)=0.
\end{equation}
Now, if we put $\lambda=q\in\mathbb{Q}$, knowing that $\sigma_i(q)=q$, we get
\begin{equation}\label{upper 2x2 on rational}
    (e_{\ell}^\ast+qe_j^\ast)P^\ast P(qe_{\ell}-e_j)=0
\end{equation}
for every rational number $q$. Denote elements of positive matrix by $P_i^\ast P_i:=(a_{\ell j})_{1\le \ell,j\le n}$.
From \eqref{upper 2x2 on rational} we have that 
$$\overline{a_{\ell j}}q^2+(a_{\ell\ell}-a_{jj})q-a_{\ell j}=0,$$
so $a_{\ell j}=0$ and $a_{\ell\ell}=a_{jj}$. This proves that $P_i^\ast P_i$ is a scalar multiple of identity and returning to the equality \eqref{upper 2x2}, we conclude that 
$$\sigma_i(\overline{\lambda})=\overline{\sigma_i(\lambda)}.$$
So, $\sigma_i$ is the identity or conjugation map. 
\end{proof}

\begin{lemma}\label{lemma:dagger}
Let $\mathfrak{A}=\bigoplus_1^k M_{n_i}(\CC)$ and let $J\subseteq\{1,\dots,k\}$. Then, a real-linear map $\dagger\colon \mathfrak{A}\to  \mathfrak{A}$, defined by 
$$\begin{cases}
A_i^\dagger=A_i;&\hbox{ if $i\in J$  }\\
A_i^\dagger =\overline{A_i}; &\hbox{ otherwise}
\end{cases}.$$
preserves mutual strong BJ orthogonality.
\end{lemma}
\begin{proof} This is an easy consequence of the fact that $\|(A_1,\dots A_k)\|=\max\|A_i\|=\max\|\overline{A_i}\|$ (the last identity follows from SVD of $A_i$).
\end{proof}

We next consider several subcases of low dimensional blocks separately. The first one is  an abelian $C^\ast$-algebra $\mathfrak{A}=\C^k,k\ge3$. Recall that in such $\mathfrak{A}$ the additive mutual strong BJ preserver will still permute the blocks, but at the moment we have not yet found its structure within each block.

\subsection{Case \texorpdfstring{$\C^k,k\ge3$}{TEXT}}

\begin{lemma}\label{Case C^k}
    Let $\Phi=(\varphi_1,\dots,\varphi_k)\colon \C^k\to \C^k, k\geq3$ be an additive map that preserves a mutual strong BJ orthogonality in one direction. Then $$\Phi(\lambda_1,\dots,\lambda_k)=\gamma u(\lambda_1,\dots,\lambda_k)^{\dagger},$$
    where $\gamma>0$, $u\in\C^k$ is a unitary and $a^{\dagger}$ is defined in Theorem \ref{THM:Main}.
\end{lemma}
\begin{proof}
    We will divide the proof into several steps.\medskip

  \noindent  {\bf Step 1.} \textit{$|\varphi_1(x)|=\dots=|\varphi_k(x)|$ for every $x\in\C$.}

    To see this note that $\|(x_1,\dots,x_k)\|=\max\|x_i\|$. It hence  follows by definition of strong BJ orthogonality that $(x,\dots x,0_{i},x,\dots, x)\strongperp_{BJ}\, (x,\dots,x,0_{j},x,\dots, x)$, whenever $i\neq j$. 
    Applying $\Phi$ this gives
    $$(\varphi_{1}(x),\dots,0_{i},\dots, \varphi_{k}(x))\strongperp_{BJ}\, (\varphi_{1}(x),\dots,0_{j},\dots, \varphi_{k}(x)),$$
    from where we deduce
    $$|\varphi_j(x)|\ge|\varphi_{\ell}(x)| \quad \text{\ for all $\ell\neq i$}\quad\hbox{  and\ }\quad |\varphi_i(x)|\ge|\varphi_{\ell}(x)| \quad\text{\ for all $\ell\neq j$}.$$
    Continuing in the same manner, by taking all possible $i\neq j$, we conclude that $|\varphi_1(x)|=|\varphi_2(x)|=\dots=|\varphi_k(x)|$ for every $x\in\C$.\medskip

  \noindent  {\bf Step 2.} \textit{If $\lambda,\mu\in\C$ and $|\lambda|=|\mu|$ then $|\varphi_i(\lambda)|=|\varphi_j(\mu)|$ for every $1\le i,j\le k$.}

    To see this, note that $(\lambda,\mu,0,\dots,0)\strongperp_{BJ}\,(0,\mu,\mu,\dots,\mu)$, so applying $\Phi$ we get $$(\varphi_1(\lambda),\varphi_2(\mu),0,\dots,0)\strongperp_{BJ}\,(0,\varphi_2(\mu),\varphi_3(\mu),\dots,\varphi_k(\mu)),$$
    and we can conclude $|\varphi_1(\lambda)|\ge|\varphi_2(\mu)|$. Now, if we change $(0,\mu,\mu,\dots,\mu)$ to $(\lambda,0,\lambda,\dots,\lambda)$, we also get $|\varphi_2(\mu)|\ge|\varphi_1(\lambda)|$, so $|\varphi_2(\mu)|=|\varphi_1(\lambda)|$. By Step 1, $|\varphi_i(\lambda)|=|\varphi_1(\lambda)|=|\varphi_2(\mu)|=|\varphi_j(\mu)|$ for every $1\le i,j\le k$, as claimed.
\medskip

  \noindent  
    {\bf Step 3.} \textit{If $\lambda,\mu\in\C$ and $|\lambda|>|\mu|$ then $|\varphi_j(\lambda)|\ge|\varphi_i(\mu)|$ for every $1\le i,j\le k$.}

    To see this, note that $(\lambda,\mu,0,\dots,0)\strongperp_{BJ}\,(0,\mu,\mu,\dots,\mu)$, so applying $\Phi$ we get 
    $$(\varphi_1(\lambda),\varphi_2(\mu),0,\dots,0)\strongperp_{BJ}\,(0,\varphi_2(\mu),\varphi_3(\mu),\dots,\varphi_k(\mu)),$$
    so $|\varphi_1(\lambda)|\ge|\varphi_2(\mu)|$, which by Step 2 proves the claim. \medskip

  \noindent   {\bf Step 4.} \textit{$\varphi_i$ is continuous for every $1\le i\le k$.}
    
    To see this, choose $\varepsilon>0$ and let $n\in\mathbb{N}$ be such that $\frac{|\varphi_i(1)|}{n}<\varepsilon$. Define $\delta=\frac{1}{n}$. Then if $|\alpha-\beta|<\delta$, using additivity, the fact that $|n(\alpha-\beta)|<1$ and Step 3, respectively, we get 
    $$|\varphi_i(\alpha)-\varphi_i(\beta)|=\frac{1}{n}\left|\varphi_i(n(\alpha-\beta))\right|\leq\frac{1}{n}|\varphi_i(1)|<\varepsilon.$$
  \noindent      {\bf Step 5.} \textit{$\varphi_j(a+bi)=a\varphi_j(1)+b\varphi_j(i)$ for every $1\le j\le k$ and every $a,b\in\mathbb{R}$.}

    To see this, let $\gamma_j:=\varphi_j(1)$. By additivity, $\varphi_j(a)=a\cdot\gamma_j$ for every $a\in\mathbb{Q}$. Then, by continuity (Step 4) $\varphi_j(a)=a\cdot\gamma$ for every $a\in\mathbb{R}$. Similarly, defining $\delta_j:=\varphi_j(i)$ and using additivity and continuity, we get $\varphi_j(bi)=bi\cdot\delta_j$. The statement then follows directly from the additivity of the mapping $\varphi_j$.\medskip

    \noindent{\bf Step 6.} \textit{$\varphi_j(\lambda)=\varphi_j(1)\lambda$ for every $\lambda\in\C$ or $\varphi_j(\lambda)=\varphi_j(1)\overline{\lambda}$ for every $\lambda\in\C$.}

    To see this, by Step~2, $|\varphi_j(\lambda)|$ is constant for every $|\lambda|=1$. Then by Step~5 $$|c\cdot\gamma_j\pm\sqrt{1-c^2}\cdot\delta_j|$$ is constant for every $c\in[-1,1]$. If $\gamma_j=0$, then $\delta_j=0$, so $\varphi_j\equiv0$, which contradicts surjectivity. Hence, 
    $\left|c\pm\sqrt{1-c^2}\frac{\delta_j}{\gamma_j}\right|$
    is constant for every $c\in[-1,1]$. By inserting $c=1$, we get that this constant is equal to 1, and hence, $\left|c\pm\sqrt{1-c^2}\frac{\delta_j}{\gamma_j}\right|=1$ for every $c\in[-1,1]$. Inserting $c=0$ gives further
    \begin{equation}\label{delta_gamma_1}
        |\delta_j|=|\gamma_j|.
    \end{equation}
    Finally, inserting $c=\frac{1}{\sqrt{2}}$, gives 
    \begin{equation}\label{delta_gamma_2}
        \left|1\pm\tfrac{\delta_j}{\gamma_j}\right|=\sqrt{2}.
    \end{equation}
    It follows from \eqref{delta_gamma_1} and \eqref{delta_gamma_2} that $\frac{\delta_j}{\gamma_j}=\pm i$. Consequently, $\varphi_j(a+bi)=a\gamma_j+bi\gamma_j =(a+bi)\gamma_j$ or $\varphi_j(a+bi)=a\gamma_j-bi\gamma_j =(a-bi)\gamma_j$ for every $a,b\in\mathbb{R}$, which was the desired result.

    \noindent{\bf Step 7.} \textit{Conclusion of the proof.}
    
    Define $\gamma=|\varphi_1(1)|>0$ and $u=\left(\frac{\varphi_1(1)}{|\varphi_1(1)|},\frac{\varphi_2(1)}{|\varphi_1(1)|},\dots,\frac{\varphi_k(1)}{|\varphi_1(1)|}\right)$, which is unitary by Step 1. Then, by Step 6, the map $\Phi$ has the desired form.
\end{proof} 

\subsection{Case \texorpdfstring{$\C\oplus M_{n_j}(\C),n_j\ge2$}{TEXT}.}
\begin{lemma}\label{Case_C_plus_M}
Let $\Phi=(\varphi_1,\varphi_2)\colon \C\oplus M_{n_j}(\C)\to \C\oplus M_{n_j}(\C),n_j\ge2$, where $\varphi_2(A)=AQ$ for some invertible matrix $Q\in M_{n_j}(\C)$. If $\Phi$ is an additive mutual strong BJ orthogonality preserver, then $Q$ is a scalar multiple of a unitary and $\varphi_1(\lambda)=\varphi_1(1)\lambda$ for every $\lambda\in\C$ or $\varphi_1(\lambda)=\varphi_1(1)\overline{\lambda}$ for every $\lambda\in\C$. In addition, $\|Q\|=|\varphi_1(1)|$.   
\end{lemma}
\begin{proof} Note that for all $\lambda\in\C, |\lambda|\le1$  and $x\in\C^2, \|x\|=1$ the operator $(\lambda,x\otimes x)$ attains its norm at $(0,x)$, so by Theorem~\ref{THEOREM} we have       
$$(\lambda,x\otimes x)\strongperp_{BJ}\,(\lambda,x'\otimes x'),$$
    where $x'$ denotes any vector perpendicular to $x$, with $\|x'\|=\|x\|$.  Applying $\Phi=(\varphi_1,\varphi_2)$, we get  
    $$(\varphi_1(\lambda),(x\otimes x) Q)\strongperp_{BJ}\,(\varphi_1(\lambda),(x'\otimes x')Q).$$
    It follows, again by Theorem \ref{THEOREM} that
    \begin{equation}\label{Qlambda}
        \|Q^{\ast}x\|=\|(x\otimes x)Q\|\ge|\varphi_1(\lambda)|.
    \end{equation}
    
    Assume, in addition $|\lambda|=1$. Notice that $(0,I)$ attains its norm on $(0,x')$ while $(\lambda,x\otimes x)$ attains its norm on $(1,0)$, so from Theorem \ref{THEOREM} 
    $$(\lambda,x\otimes x)\strongperp_{BJ}\,(0,I).$$  
    Applying $\Phi$ we get
    $$(\varphi_1(\lambda),(x\otimes x)Q)\strongperp_{BJ}\,(0,Q),$$
    so by the definition of strong BJ orthogonality, keeping in mind that $Q$ is invertible, we necessarily have $$|\varphi_1(\lambda)|\ge\|(x\otimes x)Q\|=\|Q^{\ast}x\|.$$ Together with \eqref{Qlambda}, it follows that $\|Q^\ast x\|=|\varphi_1(\lambda)|$ for all unit vectors $x\in\C^2$ and $|\lambda|=1$. Consequently, $Q^{\ast}$ is a scalar multiple of a unitary and $|\varphi_1|$ is constant on a unit circle.
    
    Moreover, \eqref{Qlambda} now implies $|\varphi_1(\lambda)|\le\|Q^\ast x\|$ for all $|\lambda|\leq1$ with an equality for $|\lambda|=1$. Hence, as in the proof of Steps 4-6 of Lemma \ref{Case C^k}, $\varphi_1$ is continuous, additive, and actually of the stated form.
\end{proof}

\subsection{Case \texorpdfstring{$M_{n_{\ell}}(\C)\oplus M_{n_j}(\C),n_{\ell},n_j\ge2$}{TEXT}.}
\begin{lemma}\label{Case_M_plus_M}
    Let $\Phi=(\varphi_1,\varphi_2)\colon M_{n_{\ell}}(\C)\oplus M_{n_j}(\C)\to M_{n_{\ell}}(\C)\oplus M_{n_j}(\C), n_{\ell},n_j\ge2$, where $\varphi_i(A)=AQ_i$ for some invertible matrices $Q_1\in M_{n_{\ell}}(\C)$ and $Q_2\in M_{n_j}(\C)$. If $\Phi$ is a mutual strong BJ orthogonality preserver, then  $Q_1$ and $Q_2$ are scalar multiples of  unitaries and $\|Q_1\|=\|Q_2\|$.
\end{lemma}
\begin{proof}
    Note that for all unit vectors $x\in\mathbb{C}^{n_{\ell}}$,  the operator $(x\otimes x,I)$ attains its norm on $(x,0)$, so by Theorem \ref{THEOREM} we have 
    \begin{equation}
        (x\otimes x,I)\strongperp_{BJ}\, (x'\otimes x',I).
    \end{equation}
    Applying $\Phi$ we get
    $$((x\otimes x)Q_1,Q_2)\strongperp_{BJ}((x'\otimes x')Q_1,Q_2),$$
    which, again by Theorem \ref{THEOREM} and invertibility of $Q_2$ implies that
    $$\|x\otimes Q_1^{\ast}x\|\ge\|Q_2\|.$$
    Therefore, 
    $$\|Q_1\|=\|Q_1^{\ast}\|\ge\|Q_1^\ast x\|=\|x\otimes Q_1^{\ast}x\|\ge\|Q_2\|,$$
    holds for all unit vectors $x\in\C^{n_{\ell}}$. Moreover, applying $\Phi$ on a pair 
    $$(I,x\otimes x)\strongperp_{BJ}\, (I,x'\otimes x'),$$
    for unit vector $x\in\C^{n_j}$, we obtain analogously that $$\|Q_2\|\ge\|Q_2^\ast x\|\ge\|Q_1\|$$
    holds for all unit vectors $x\in\C^{n_j}$. Consequently, $\|Q_1\|=\|Q_2\|$, and both $Q_1$ and $Q_2$ are scalar multiples of unitaries. 
\end{proof}

\begin{proof}[Proof of Theorem \ref{THM:Main}]
    As mentioned in \eqref{finite-dimensional} we may assume that $\mathfrak{A}$ is a block-diagonal matrix algebra.
    Now, by Lemma \ref{singula_to_singular} $\Phi$ preserves singularity and hence by Theorem \ref{singularity_preservers} there exists a permutation of blocks $\pi$, which respects their dimension, such that 
    $$\Phi(M_{n_i}(\C))=\pi(M_{n_i}(\C)).$$ Notice that $\pi$ is a linear $\ast$-isometry, so it preserves mutual strong  BJ orthogonality in both direciotns. Therefore, by composing $\Phi$ with $\pi^{-1}$ we obtain an additive surjection (denoted again by $\Phi$), which still preserves mutual strong BJ orthogonality, and satisfies $$\Phi(M_{n_i}(\C))=M_{n_i}(\C).$$ Hence, we may decompose
    $$\Phi=\Phi_1\oplus\dots\oplus\Phi_k.$$
    Assume $n_i\ge2$. Then, by Theorem \ref{singularity_preservers}, there exists a field endomorphism $\sigma_i\colon\C\to\C$ and invertible matrices $P_i,Q_i\in M_{n_i}(\C)$, such that 
    $$\Phi_i\colon A_i\mapsto P_iA_i^{\sigma_i} Q_i\quad \hbox{ or }\quad \Phi_i\colon A_i\mapsto P_i(A_i^T)^{\sigma_i} Q_i.$$
    The second case is impossible by Lemma \ref{tranpose_impossible}. In the first case, by Lemma \ref{P_and_sigma}, $P_i=\lambda_iU_i$ is a scalar multiple of unitary and $\sigma_i$ is identity or complex conjugation. Hence if we replace $\Phi$ by a map $U^{\ast}\Phi$, for a unitary $U=\bigoplus_{i=1}^{k}U_i$, with $U_{\ell}=1$ if $n_{\ell}=1$, we can and will assume that 
    $$\Phi_i\colon A_i\to A_i^{\dagger}(\lambda_i Q_i);\quad A_i\in M_{n_i}(\C)\text{\ with\ }n_i\ge2.$$

    Assume $n_i=1$. By the assumptions on $\mathfrak{A}$ we have that $\mathfrak{A}$ is either abelian with at least three blocks, or else there exists a block with $n_j\ge2$. In the first case, by Lemma \ref{Case C^k} we imediately get a form as stated in (ii) of Theorem \ref{THM:Main}. In the second case, $\Phi_i$ is a scalar multiple of identity on $\C$ or a scalar multiple of complex conjugation on $\C$, while $Q_j$ is a scalar multiple of a unitary.
    
    By Lemma \ref{lemma:dagger}, if we compose the additive~$\Phi=\bigoplus\Phi_i$ with a map $\Psi\colon\mathfrak{A}\to\mathfrak{A}$, which is identity on blocks with $\Phi_i$ linear, and entrywise conjugation on blocks with $\Phi_i$ conjugate-linear, we will get a linear preserver of mutual strong BJ orthogonality. Clearly, $\Phi_i$ must also preserves mutual strong BJ orthogonality. Now, we are in a position to apply a known result \cite[Theorem 2.7]{StrongMutualPreservers} by which $Q_i=\mu_iV_i$ is a scalar multiple of a unitary if $n_i\ge3$. The same conclusion holds also if $n_i=2$ by Lemma \ref{Case_M_plus_M}. We may hence replace $\Phi$ by $\Phi V^\ast$ for a unitary $V=\bigoplus_{i=1}^{k}V_i$, with $V_{\ell}=1$ if $n_{\ell}=1$ to achieve that $$\Phi(A_1\oplus\dots\oplus A_k)=(\lambda_1\mu_1A_1\oplus\dots\oplus\lambda_k\mu_kA_k).$$
    By the last statements in Lemmas \ref{Case_C_plus_M} and \ref{Case_M_plus_M} we conclude that $$|\lambda_1\mu_1|=\dots=|\lambda_k\mu_k|=\gamma>0.$$ 
    Hence $$\Phi(A_1\oplus\dots\oplus A_k)=\gamma U(A_1\oplus\dots\oplus A_k),$$ 
    for a unitary $U=\left(\frac{\lambda_1\mu_1}{\gamma}I_{n_1},\dots,\frac{\lambda_k\mu_k}{\gamma}I_{n_k}\right)$.
\end{proof}

\section{Concluding remarks}

Let us show by examples  that the conclusions of Theorem \ref{THM:Main} are not valid if its assumptions fail. Recall than if  $\mathfrak{A}=\C\oplus\C$, then a nonstandard (nonsurjective) linear preserver of mutual strong  BJ orthogonality was already given in Example~\ref{exam:32}.

\begin{example}
    If $\mathfrak{A}=\C$, then $a\perp_{BJ}^S b$ if and only if one of $a,b\in\CC$ vanishes, which is, therefore, equivalent to  $a\strongperp_{BJ}b$. Here, every additive map on $\CC$ preserves a mutual strong BJ orthogonality.
\end{example}

\begin{example}
    If $\mathfrak{A}=\C\oplus\C$, an additive surjection $\Phi\colon\mathfrak{A}\to\mathfrak{A}$ defined by $$\Phi(\lambda,\mu)=(\lambda,2\mu)$$ preserves a mutual strong BJ orthogonality, but it is not of the form stated in Theorem \ref{THM:Main}. 
\end{example}

\begin{example}
   If $\mathfrak{A}=M_2(\C)$, and an additive surjection $\Phi\colon\mathfrak{A}\to\mathfrak{A}$ preserves mutual strong BJ orthogonality,   then by Lemma \ref{singula_to_singular} and Theorem \ref{singularity_preservers} we have $\Phi(A)=PA^{\sigma}Q$ or $\Phi(A)=P(A^T)^{\sigma}Q$ for some invertible $P,Q\in M_2(\C)$ and field isomorphism $\sigma\colon\C\to\C$. The second case is impossible by Lemma \ref{tranpose_impossible}. By Lemma \ref{P_and_sigma} $P$ is a scalar multiple of a unitary operator and $\sigma$ is an identity or conjugation. If needed, we can replace $\Phi$ by $X\mapsto \overline{\Phi(X)}$ (entrywise conjugation) to achieve that $\Phi$ is linear. Finally (this part is actually from \cite[Theorem 2.9.]{StrongMutualPreservers}, $\Phi(A)=PA Q$ is a mutual strong BJ preserver for every invertible $Q$.
\end{example}

\begin{remark}
    If we do not impose additivity, there will be wild-types of strong BJ orthogonality preservers; see \cite[Theorem 3.27 and Example 3.30]{KST}.
\end{remark}

\section{Acknowledgements}
Part of this work was carried out during the visit of S. Stefanović to the University of Primorska within the ERASMUS+ project, which he thanks for the hospitality.

\bibliographystyle{abbrv}
\bibliography{BJ_orthogonality_preservers}

\end{document}